\documentclass[paper=a4,english,fontsize=11pt,parskip=half,abstract=true]{scrartcl}
\usepackage{babel}
\usepackage[utf8]{inputenc}
\usepackage[T1]{fontenc}
\usepackage[left=20mm,right=20mm,top=30mm,bottom=30mm]{geometry}
\usepackage{amsmath}
\usepackage{amsthm}
\usepackage{amssymb}
\usepackage{enumerate} 
\usepackage{thmtools}
\usepackage{mathtools}
\mathtoolsset{centercolon} 
\usepackage[bookmarks=true,
            pdftitle={Groups with supersolvable automorphism group},
            pdfauthor={Benjamin Sambale},
            pdfkeywords={},
            pdfstartview={FitH}]{hyperref}

\newtheorem{Thm}{Theorem} 
\newtheorem{Lem}[Thm]{Lemma}

\newtheorem{Cor}[Thm]{Corollary}
\newtheorem{Pro}[Thm]{Problem}
\theoremstyle{definition}
\newtheorem{Def}[Thm]{Definition}

\numberwithin{equation}{section}

\allowdisplaybreaks[1]

\renewcommand{\phi}{\varphi}
\newcommand{\C}{\mathrm{C}}

\newcommand{\Z}{\mathrm{Z}}
\newcommand{\pcore}{\mathrm{O}}

\newcommand{\ZZ}{\mathbb{Z}}

\newcommand{\NN}{\mathbb{N}}
\newcommand{\FF}{\mathbb{F}}
\newcommand{\Aut}{\mathrm{Aut}}
\newcommand{\Inn}{\mathrm{Inn}}

\newcommand{\GL}{\operatorname{GL}}

\title{Groups with supersolvable\\ automorphism group}
\author{Benjamin Sambale\footnote{Institut für Algebra, Zahlentheorie und Diskrete Mathematik, Leibniz Universität Hannover, Welfengarten 1, 30167 Hannover, Germany,
\href{mailto:sambale@math.uni-hannover.de}{sambale@math.uni-hannover.de}}}
\date{\today}

\begin{document}
\frenchspacing
\maketitle
\begin{abstract}\noindent
We call a finite group $G$ ultrasolvable if it has a characteristic subgroup series whose factors are cyclic. It was shown by Durbin--McDonald that the automorphism group of an ultrasolvable group is supersolvable. The converse statement was established by Baartmans--Woeppel under the hypothesis that $G$ has no direct factor isomorphic to the Klein four-group. We extend this result by proving that $\Aut(G)$ is supersolvable if and only if $G$ is ultrasolvable or $G=H\times C_2\times C_2$ where $H$ is ultrasolvable of odd order. This corrects an erroneous claim by Corsi Tani. Our proof is more elementary than Baartmans--Woeppel's and uses some ideas of Corsi Tani and Laue.
\end{abstract}

\textbf{Keywords:} automorphism groups, supersolvable groups, characteristic series\\
\textbf{AMS classification:} 20D45, 20F16, 20F22, 20F28 

\section{Introduction}

A finite group $G$ is solvable if and only if it has a subnormal series 
\[1=G_0\unlhd G_1\unlhd\ldots\unlhd G_n=G\]
such that the factors $G_i/G_{i-1}$ are cyclic for $i=1,\ldots,n$. Moreover, $G$ is called supersolvable if there is a analogous \emph{normal} series (i.\,e. $G_i\unlhd G$ for $i=1,\ldots,n$) with cyclic factors. It is therefore natural to investigate groups $G$ with a characteristic series (i.\,e. $\alpha(G_i)=G_i$ for all $\alpha\in\Aut(G)$ and $i=1,\ldots,n$) with cyclic factors. 
Durbin--McDonald~\cite{DurbinMcDonald} have called these groups c.c.s. groups (\emph{characteristic cyclic series}), but we like to call them \emph{ultrasolvable}. It was shown in \cite[Theorem~1]{DurbinMcDonald} that the automorphism group of an ultrasolvable group is supersolvable. We provide an elementary proof in the next section for the convenience of the reader (\autoref{lemcharseries}).

The characteristically simple Klein four-group $V\cong C_2\times C_2$ with supersolvable automorphism group isomorphic to the symmetric group $S_3$ shows that the converse does not hold in general. Nevertheless, Durbin--McDonald have conjectured that the converse holds under the hypothesis that $G$ has no direct factor isomorphic to $V$. This was eventually proven by Baartmans--Woeppel~\cite[Theorem~C]{BaartmansWoeppel} relying on deep theorems of Baer~\cite{Baer}. Corsi Tani~\cite[Theorema~3]{CorsiTani} showed that a $p$-group $P\not\cong V$ is ultrasolvable if and only if $\Aut(P)$ is supersolvable. (A more precise description of the automorphism group in this case has been obtained by Lakatos~\cite{Lakatos}.)
It is claim in \cite[footnote (**) on p.~106]{CorsiTani} (and in its MathSciNet review \href{https://mathscinet.ams.org/mathscinet/relay-station?mr=670814}{MR0670814}) that the same result holds for arbitrary finite groups not isomorphic of $V$. But this is plainly false with $G=C_6\times C_2$ being a counterexample ($\Aut(G)\cong S_3\times C_2$). 

The aim of this paper is the following characterization of groups with supersolvable automorphism group.

\begin{Thm}\label{main}
For every finite group $G$, the following statements are equivalent:
\begin{enumerate}[(1)]
\item $\Aut(G)$ is supersolvable.
\item $G$ is ultrasolvable or $G\cong H\times C_2\times C_2$ where $H$ is ultrasolvable of odd order.
\end{enumerate}
\end{Thm}

The proof is mostly self-contained and makes use of ideas of Corsi Tani and Laue~\cite{Laue}. 

\section{\texorpdfstring{$A$}{A}-solvable groups}

Our notation is standard and follows Kurzweil--Stellmacher's book~\cite{Kurzweil}. From now on, $G$ will always be a finite group.

\begin{Def}
A group $A$ acts on $G$ via a homomorphism $A\to\Aut(G)$ (in most cases we consider $A\le\Aut(G)$ with the embedding homomorphism). A subnormal series $1=G_0\unlhd G_1\unlhd\ldots\unlhd G_n=G$ is called an $A$-\emph{series} if each $G_i$ is $A$-invariant, i.\,e. $\alpha(G_i)=G_i$ for all $\alpha\in A$. 
If $G$ has an $A$-series with cyclic factors $G_i/G_{i-1}$ for $i=1,\ldots,n$, then $G$ is called $A$-\emph{solvable}. 
\end{Def}

Note that $G$ is $1$-solvable, $\Inn(G)$-solvable or $\Aut(G)$-solvable if and only if $G$ is solvable, supersolvable or ultrasolvable respectively ($\Inn(G)$ denotes the inner automorphism group of $G$). 
An $A$-series of $G$ with cyclic factors can be refined to an $A$-series with factors of prime order, because subgroups of cyclic groups are characteristic. This will often be used in the following. In the usual manner, one verifies that $A$-invariant subgroups of $A$-solvable groups are $A$-solvable. The same holds for quotients by $A$-invariant normal subgroups.

We first prove \cite[Theorem~1]{DurbinMcDonald} with a simpler argument.

\begin{Thm}[\textsc{Durbin--McDonald}]\label{lemcharseries}
If $G$ is ultrasolvable, then $\Aut(G)$ is supersolvable.
\end{Thm}
\begin{proof}
Let $A:=\Aut(G)$. 
Let $1=G_0\unlhd\ldots\unlhd G_n=G$ be an $A$-series with cyclic factors. 
Let $B\unlhd A$ be the kernel of the action of $A$ on $G_{n-1}$. Let $C\unlhd A$ be the kernel of the action of $A$ on $G/G_1$. By induction on $|G|$, we may assume that $A/B\le\Aut(G_{n-1})$ and $A/C\le\Aut(G/G_1)$ are supersolvable. Then $A/(B\cap C)\le A/B\times A/C$ is supersolvable. It suffices to show that $D:=B\cap C$ is cyclic. Let $G/G_{n-1}=\langle x\rangle G_{n-1}$ and $G_1=\langle y\rangle$. We choose $\alpha\in D\setminus\{1\}$ with $\alpha(x)=xy^a$ such that $a\ge 0$ is as small as possible. 
Since $\alpha$ is uniquely determined by $\alpha(x)$, we must have $a\ge 1$. Let $\beta\in D$ be arbitrary with $\beta(x)=xy^b$. By Euclidean division, there exist $q,r\in\ZZ$ such that $b=qa+r$ and $0\le r<a$. Since
\[(\beta^{-1}\alpha^q)(x)=\beta^{-1}(x)y^{qa}=xy^{qa-b}=xy^r,\]
it follows that $r=0$ and $\beta=\alpha^q$. This shows that $D=\langle\alpha\rangle$. 
\end{proof}

The next lemma is well-known, but we include a proof for sake of completeness. As usual, $\pcore_{\pi}(G)$ denotes the largest normal $\pi$-subgroup of $G$. 

\begin{Lem}[Sylow tower property]\label{lempi}
Let $G$ be supersolvable and $n\in\NN$. Let $\pi$ be the set of primes $p\ge n$. Then $G$ has a normal Hall $\pi$-subgroup. 
In particular, $G$ is $2$-nilpotent, i.\,e. $G$ has a normal $2$-complement.
\end{Lem}
\begin{proof}
We argue by induction on $|G|$. Let $N\unlhd G$ be a minimal normal subgroup. Then $N$ is a chief factor and therefore $q:=|N|$ is a prime. By induction $G/N$ has a normal Hall $\pi$-subgroup $K/N\le G/N$. If $q\ge n$, then $K$ is a normal Hall $\pi$-subgroup of $G$. Thus, let $q<n$. Then $N$ has a complement $L$ in $K$ by the Schur--Zassenhaus theorem. Since $|\Aut(N)|=q-1$ has no prime divisor in $\pi$, we obtain $N\le\Z(K)$ and $K=N\times L$. It follows that $L=\pcore_\pi(K)=\pcore_\pi(G)$ is a normal Hall $\pi$-subgroup of $G$. 
The second claim follows with $n=3$. 
\end{proof}

The following result generalizes \cite[Lemma~1]{CorsiTani} and is related to the uniqueness in the Krull--Remak--Schmidt theorem (see \cite[I.12.6]{Huppert}). 

\begin{Lem}\label{lemchardir}
Let $G=A\times B$ with $\gcd(|A/A'|,|\Z(B)|)\ne 1$. Then $A$ is not characteristic in $G$.
\end{Lem}
\begin{proof}
By hypothesis, there exists a non-trivial homomorphism $\phi:A\to A/A'\to\Z(B)$. It is easy to check that $\psi:G\to G$, $(a,b)\mapsto(a,\phi(a)b)$ is an automorphism with $\psi(A)\ne A$. 
\end{proof}

\begin{Cor}\label{corchardir}
Let $P=A\times B$ be a $p$-group such that $A$ is characteristic in $P$. Then $A=1$ or $B=1$.
\end{Cor}

Now we prove \cite[Lemma~2]{CorsiTani}, which is a consequence of Laue~\cite[Satz~3]{Laue}. 

\begin{Lem}\label{lempcore}
For every non-abelian $p$-group $P$, we have $\pcore_{p'}(\Aut(P))=1$.
\end{Lem}
\begin{proof}
Let $A:=\pcore_{p'}(\Aut(P))$. Then the normal $p$-subgroup $P/\Z(P)\cong\Inn(P)\unlhd\Aut(P)$ is centralized by $A$. It follow that $[P,A]\le\Z(P)<P$, since $P$ is non-abelian.
From the theory of coprime actions (see \cite[8.2.7]{Kurzweil}), we obtain $P=[P,A]\C_P(A)$ and $[P,A]=[P,A,A]\le[\Z(P),A]$. Since $\Z(P)$ is abelian, we further know that
\[[P,A]\cap\C_P(A)=[\Z(P),A]\cap\C_{\Z(P)}(A)=1\]
by \cite[8.4.2]{Kurzweil}.
Therefore, $P=[P,A]\times\C_P(A)$ is a decomposition into characteristic subgroups, because $A\unlhd\Aut(P)$. \autoref{corchardir} implies $P=\C_P(A)$ and $A=1$.
\end{proof}

\section{Strictly \texorpdfstring{$p$}{p}-closed groups}

The following definition goes back to Baer~\cite{Baer}.

\begin{Def}
Let $p$ be a prime. Then $G$ is called \emph{strictly $p$-closed} if $G$ has a normal Sylow $p$-subgroup $P$ such that $G/P$ is abelian of exponent dividing $p-1$, i.\,e. $x^{p-1}\in P$ for all $x\in G$. 
\end{Def}

It is a routine exercise to show that subgroups and quotients of strictly $p$-closed groups are strictly $p$-closed. 
Moreover, $G$ is strictly $p$-closed if and only if $G/\pcore_p(G)$ is strictly $p$-closed (see also \autoref{lemstrp2} below).
In particular, every $p$-group is strictly $p$-closed, and for $2$-groups the converse is also true.

\begin{Lem}\label{lemschur}
Let $G$ be strictly $p$-closed. If $G$ acts irreducibly on an elementary abelian $p$-group $V$, then $|V|=p$.
\end{Lem}
\begin{proof}
We may assume that $G$ acts faithfully on $V$. 
Let $P$ be the unique Sylow $p$-subgroup of $G$. 
By orbit counting, we have $U:=\C_V(P)\ne 1$. Since $P\unlhd G$, $U$ is $G$-invariant. Since $V$ is irreducible, $U=V$ and $P=1$. Hence, $G$ is abelian of exponent dividing $p-1$. By Schur's lemma, $G$ is cyclic, say $G=\langle x\rangle$ (see \cite[8.3.3 or 8.6.1]{Kurzweil}). Let $f\in\GL(V)$ be the linear map induced by $x$. Since $x^{p-1}=1$, the minimal polynomial of $f$ divides $X^{p-1}-1$. It follows that $f$ has an eigenvalue in $\FF_p^{\times}$. A corresponding eigenvector generates a $G$-invariant subspace of dimension $1$. Therefore, $|V|=p$ (this also follows directly from \cite[8.6.1(b)]{Kurzweil}). 
\end{proof}

Now we are in a position to prove an extension of \cite[Teorema~8(b)]{CorsiTani}, which is related to \cite[Theorem~2.1]{Baer}.
As usual, we denote the Frattini subgroup of $G$ by $\Phi(G)$.

\begin{Thm}\label{thmpgrpiff}
Let $P$ be a $p$-group and $A\le\Aut(P)$. Then the following statements are equivalent:
\begin{enumerate}[(i)]
\item $P$ is $A$-solvable.
\item $P/\Phi(P)$ is $A$-solvable.
\item $A$ is strictly $p$-closed.
\end{enumerate}
\end{Thm}
\begin{proof}
Suppose first that $P$ is $A$-solvable.  
Then the normal series $\Phi(P)\le P$ can be refined to an $A$-series $\Phi(P)=P_0\unlhd P_1\unlhd\ldots\unlhd P_n=P$ such that $|P_i/P_{i-1}|=p$ for $i=1,\ldots,n$ using the Jordan--Hölder theorem for operator groups (see \cite[1.8.1]{Kurzweil}). 
Hence, $P/\Phi(P)$ is $A$-solvable. Now given the $A$-series as above, $A$ acts on $\bigtimes_{i=1}^nP_i/P_{i-1}$ fixing each factor. The kernel $B\unlhd A$ of this action is a $p$-group by a theorem of Burnside (see \cite[8.2.9]{Kurzweil}). Since
\[A/B\le\bigtimes_{i=1}^n\Aut(P_i/P_{i-1})\cong C_{p-1}^n\] 
is abelian of exponent dividing $p-1$, we conclude that $A$ is strictly $p$-closed.

Suppose conversely that $A$ is strictly $p$-closed. By standard group theory, 
\[\Omega:=\Omega(\Z(P)):=\{z\in\Z(P):z^p=1\}\] 
is a characteristic non-trivial elementary abelian $p$-subgroup of $P$. By \autoref{lemschur}, there exists an $A$-invariant subgroup $Q\le\Omega$ of order $p$. By induction on $|P|$, $P/Q$ is $A$-solvable and so is $P$.
\end{proof}

Since “most” $p$-groups have no non-trivial $p'$-automorphisms, there is no hope to classify ultrasolvable ($p$-)groups.
Concrete examples are the $p$-groups of maximal nilpotency class and order $\ge p^4$ (see \cite[Hilfssätze~III.14.2, III.14.4]{Huppert}).
The following corollary is not needed in the sequel, but interesting on its own. 

\begin{Cor}
A $2$-group $P$ is ultrasolvable if and only if $\Aut(P)$ is $2$-group.
\end{Cor}

\begin{Cor}[\textsc{Baer}]\label{corbaer}
Every strictly $p$-closed group is supersolvable.
\end{Cor}
\begin{proof}
Let $G$ be strictly $p$-closed with normal Sylow $p$-subgroup $P$. Then $P$ is $G$-solvable by \autoref{thmpgrpiff}. Hence, $G$ is supersolvable.
\end{proof}

We obtain a partial converse of \autoref{corbaer}. 

\begin{Lem}\label{lemstrp1}
Let $G$ be a supersolvable group such that $\pcore_{p'}(G)=1$ for some prime $p$. Then $G$ is strictly $p$-closed.
\end{Lem}
\begin{proof}
By \autoref{lempi}, $G$ has a normal Sylow $p$-subgroup $P$ (and no other normal Sylow subgroups). By the Schur--Zassenhaus theorem (or Hall's theorem for solvable groups), $P$ has a complement $K$ in $G$. Since $\pcore_{p'}(G)=1$, $K$ acts faithfully on $P$. Moreover, $P$ is a $K$-solvable group, because $G$ is supersolvable. By \autoref{thmpgrpiff}, $K$ is strictly $p$-closed and so must be $G$.
\end{proof}

It is easy to see that the condition $\pcore_{p'}(G)=1$ is not fulfilled by strictly $p$-closed groups in general. In order to obtain an equivalent characterization, we recall the notation $\pcore_{pp'}(G)/\pcore_p(G):=\pcore_{p'}(G/\pcore_p(G))$. 

\begin{Lem}\label{lemstrp2}
For a supersolvable group $G$ the following assertions are equivalent:
\begin{enumerate}[(1)]
\item $G$ is strictly $p$-closed.
\item $\pcore_{pp'}(G)$ is strictly $p$-closed.
\item $\pcore_{pp'}(G)/\pcore_p(G)$ is abelian of exponent dividing $p-1$.
\end{enumerate}
\end{Lem}
\begin{proof}
It is clear that (1) implies (2) and (2) implies (3). Now assume (3). If $p$ does not divide $|G|$, then $G=\pcore_{p'}(G)=\pcore_{pp'}(G)$ is strictly $p$-closed by hypothesis. Hence, we may assume that $p$ divides $|G|$.
By \autoref{lempi}, $G$ has a normal Sylow $q$-subgroup $Q$, where $q$ is the largest prime divisor of $|G|$. Moreover, $Q\le\pcore_{pp'}(G)$ and therefore $q=p$. Since $G/\pcore_p(G)=G/Q$ is a $p'$-group, it follows again that $G=\pcore_{pp'}(G)$ is strictly $p$-closed. 
\end{proof}

\begin{Lem}[\textsc{Durbin--McDonald}]\label{lemabel}
Let $P=C_{p^{a_1}}\times\ldots\times C_{p^{a_n}}$ be an abelian $p$-group with $1\le a_1\le\ldots\le a_n$. Then $\Aut(P)$ is supersolvable if and only if $a_1<\ldots<a_n$ or $P=C_2\times C_2$.
\end{Lem}
\begin{proof}
If $P\cong C_2\times C_2$, then $\Aut(P)\cong\GL(2,2)\cong S_3$ is supersolvable. 
Suppose next that $a_1<\ldots<a_n$. Let $A:=\Aut(P)$. By \autoref{thmpgrpiff} and \autoref{corbaer}, it suffices to show that $P/\Phi(P)\cong C_p^n$ is $A$-solvable. 
Define characteristic subgroups 
\[P_i:=\{x\in P:x^{p^{a_i}}=1\}\Phi(P)\le P\] 
for $i=1,\ldots,n$. Then $P_0:=\Phi(P)<P_1<\ldots<P_n=P$ and $|P_i:P_{i-1}|=p$ for $i=1,\ldots,n$. 
So we are done.

Conversely, let $a_k=a_{k+1}$ for some $k$. Since $\Aut(C_p^{a_k}\times C_p^{a_k})$ is a subgroup of $A$, we may assume that $P\cong C_{p^a}\times C_{p^a}$ in order to show that $A$ is not supersolvable. Let $P=\langle x,y\rangle$ and $x',y'\in P$ such that $\{x'\Phi(P),y'\Phi(P)\}$ is a basis of the elementary abelian group $P/\Phi(P)\cong C_p\times C_p$. Recall that Burnside's basis theorem implies $P=\langle x',y'\rangle$. It is easy to check that there exists an automorphism $\gamma\in A$ such that $\gamma(x)=x'$ and $\gamma(y)=y'$. This shows that the restriction map 
\[\Gamma:A\to\Aut(P/\Phi(P))\cong\GL(2,p)\] 
is surjective. If $p\ge 5$, then $\GL(2,p)$ (and $A$ in turn) is not even solvable. For $p=3$, $\GL(2,3)\cong Q_8\rtimes C_3$ is not supersolvable by \autoref{lempi}, for instance. 
Finally, let $p=2$ and $a\ge 2$.
By Burnside's theorem mentioned before, the kernel of $\Gamma$ is a $2$-group. Thus, $|A|=2^s3$ for some $s\ge 1$. It is easy to check that the maps $\alpha,\beta:P\to P$ with
\begin{align*}
\alpha(x)&=y,&\alpha(y)&=(xy)^{-1},\\
\beta(x)&=y^{-1},& \beta(y)&=xy^{-1}
\end{align*}
are automorphisms of order $3$ and $\langle\alpha\rangle\ne\langle\beta\rangle$. In particular, $A$ does not have a normal Sylow $3$-subgroup. By \autoref{lempi}, $A$ is not supersolvable.
\end{proof}

We end this section by proving the converse of \autoref{lemcharseries} for $p$-groups.

\begin{Thm}[\textsc{Corsi Tani}]\label{thmCorsi}
Let $P\not\cong C_2\times C_2$ be a $p$-group such that $\Aut(P)$ supersolvable. Then $P$ is ultrasolvable.
\end{Thm}
\begin{proof}
Suppose first that $P$ is abelian. Then $P\cong C_{p^{a_1}}\times\ldots\times C_{p^{a_n}}$ with $a_1<\ldots<a_n$ by \autoref{lemabel}. In fact, we have shown in the proof of \autoref{lemabel} that $P/\Phi(P)$ is $A$-solvable, where $A:=\Aut(P)$. Hence, $P$ is ultrasolvable by \autoref{thmpgrpiff}.
Now assume that $P$ is non-abelian.
By \autoref{thmpgrpiff}, it suffices to show that $\Aut(P)$ is strictly $p$-closed. But this follows from \autoref{lempcore} and \autoref{lemstrp1}.
\end{proof}

It is straight-forward to deduce a characterization of supersolvable nilpotent groups from \autoref{thmCorsi}, but this will be generalized by our main theorem.

\section{Proof of \autoref{main}}

Let $V=C_2\times C_2$ be the Klein four-group.
If $G$ is ultrasolvable, then $\Aut(G)$ is supersolvable by \autoref{lemcharseries}.
If $H$ is ultrasolvable of odd order, then $\Aut(H\times V)=\Aut(H)\times S_3$ is supersolvable. 
Now assume conversely that $A:=\Aut(G)$ is supersolvable. Then $G/\Z(G)\cong\Inn(G)\le A$ is supersolvable and so is $G$.

\textbf{Case~1:} $G=H\times V$ for some $H\le G$.\\
Suppose first that $|H|$ is even. By \autoref{lempi}, $H$ is $2$-nilpotent. In particular, $\gcd(|H/H'|,|V|)\ne 1$.
As in the proof of \autoref{lemchardir}, we construct an automorphism $\alpha:G\to G$, $(h,v)\mapsto(h,\phi(h)v)$, where $\phi:H\to V$ is non-trivial. Since $V$ has exponent $2$, $\alpha$ is an involution. Let $\beta\in\Aut(V)$ be of order $3$. We extend $\beta$ to $G$ by $\beta(h)=h$ for all $h\in H$. Then $\beta\in\pcore_{2'}(A)$ and $\gamma:=[\alpha,\beta]=\alpha\beta^{-1}\alpha\beta\in \pcore_{2'}(A)$ by \autoref{lempi}. 
We compute
\[\gamma(h,v)=(\alpha\beta^{-1}\alpha)(h,\beta(v))=(\alpha\beta^{-1})(h,\phi(h)\beta(v))=\alpha(h,\beta^{-1}(\phi(h))v)=(h,\phi(h)\beta^{-1}(\phi(h))v)\]
for all $(h,v)\in G$. By construction, there exists $h\in H$ such that $w:=\phi(h)\ne 1$. Then $\beta^{-1}(w)\ne w$ and $w\beta^{-1}(w)\ne 1$. This shows that $\gamma$ has order $2$, which contradicts $\gamma\in\pcore_{2'}(A)$.
Therefore, $|H|$ is odd and the claim follows by induction on $|G|$, because $\Aut(H)\le A$ is supersolvable.

\textbf{Case~2:} $V$ is not a direct factor of $G$.\\
Let $p$ be a prime divisor of $|G|$. Let $\pi$ be the set of primes $q>p$. Then $N:=\pcore_\pi(G)$ is a normal Hall $\pi$-subgroup by \autoref{lempi}. For a Sylow $p$-subgroup $P$ of $G$ we further have that $NP=\pcore_{\pi\cup\{p\}}(G)$ is characteristic in $G$. Arguing by induction on $|G|$ (starting with the largest prime divisor $p$), it suffices to show that $\overline{P}:=PN/N\cong P$ is $A$-solvable. Equivalently, by \autoref{thmpgrpiff}, we need that the image $\overline{A}$ of $A$ in $\Aut(\overline{P})$ is strictly $p$-closed. We prove this via \autoref{lemstrp2}.
Let $Z:=P\cap\Z(G)\le\Z(P)$.
Since $A$ is supersolvable, 
\[\overline{P}/\overline{Z}\cong P\Z(G)/\Z(G)\le G/\Z(G)\cong\Inn(G)\] 
is $A$-solvable. 
Let $\overline{A}_p:=\pcore_p(\overline{A})$ and $\overline{B}:=\pcore_{pp'}(\overline{A})=\overline{A}_p\rtimes\overline{K}$ for some complement $\overline{K}$. We need to show that $\overline{K}$ is abelian of exponent dividing $p-1$.
We define the auxiliary group 
\[\widehat{B}:=\overline{P}\rtimes\overline{B}=(\overline{P}\rtimes\overline{A}_p)\rtimes\overline{K}.\] 
Since $\overline{B}$ is supersolvable and $\overline{P}/\overline{Z}$ is $\overline{B}$-solvable, there exists a $\widehat{B}$-series from $\overline{Z}$ to $\overline{P}$ to $\overline{P}\rtimes\overline{A}_p$ in $\widehat{B}$ with factors of order $p$. Let $\overline{C}\unlhd\overline{K}$ be the kernel of the action of $\overline{K}$ on the direct factors of this series. Then $\overline{K}/\overline{C}$ is abelian of exponent dividing $p-1$. Thus, it suffices to prove that $\overline{C}=1$. 
Note that $\overline{C}\Z(\overline{A}_p)=\C_{\overline{B}}(\overline{A}_p)\unlhd\overline{A}$ (using \cite[8.2.2]{Kurzweil}). In particular, $\overline{C}\unlhd\overline{A}$. 

By construction, we have
$[\overline{P},\overline{C}]\le\overline{Z}$.  
As in the proof of \autoref{lemchardir}, it follows that
\[\overline{P}=[\overline{P},\overline{C}]\times\C_{\overline{P}}(\overline{C}).\]
The preimage $P_0\le Z$ of $[\overline{P},\overline{C}]$ is a direct factor of $P$. By Gaschütz' theorem, $P_0$ has a complement $K$ in $G$ (see \cite[3.3.2]{Kurzweil}). Since $P_0\le\Z(G)$, we must have $G=P_0\times K$.
Note that $Z$ is $A$-invariant as the unique Sylow $p$-subgroup of $\Z(G)$. Moreover, $[\overline{P},\overline{C}]$ is $\overline{A}$-invariant since $\overline{C}\unlhd\overline{A}$. It follows that $\alpha(P_0)\subseteq NP_0$ for $\alpha\in A$.
Hence, $\alpha(P_0)\subseteq Z\cap NP_0=(Z\cap N)P_0=P_0$, i.\,e. $P_0$ is $A$-invariant. 

If $P_0=1$, then $\overline{C}$ acts trivially on $\overline{P}$ and therefore $\overline{C}=1$. In this case we are done by \autoref{lemstrp2}
Next, let $P_0\ne 1$. 
Since $\Aut(P_0)\le A$ is supersolvable and $P_0\not\cong V$, it follows from \autoref{thmCorsi} that $P_0$ is ultrasolvable. Since $P_0$ is $A$-invariant, it suffices to show that $G/P_0\cong K$ is ultrasolvable. Every direct factor of $K$ is also a direct factor of $G$. So by hypothesis, $V$ is not a direct factor of $K$. Since $\Aut(K)\le A$ is supersolvable, $K$ is ultrasolvable by induction. This completes the proof. 

\section{Fully solvable groups}

We close this paper with an open problem. Recall that a subgroup $H\le G$ is called \emph{fully invariant} if $\alpha(H)\le H$ for every endomorphism $\alpha:G\to G$. Obviously, fully invariant subgroups are characteristic. It is natural to call a group $G$ \emph{fully solvable} if there exists a series of fully invariant subgroups
\[1=G_0\unlhd\ldots\unlhd G_n=G\]
such that the factors $G_i/G_{i-1}$ are cyclic for $i=1,\ldots,n$. Fully solvable groups are certainly ultrasolvable. For abelian groups the converse is also true, because the characteristic subgroups constructed in the proof of \autoref{lemabel} are fully invariant (noticing that $\Phi(P)=\langle x^p:x\in P\rangle$ for every abelian $p$-group $P$).

On the other hand, the ultrasolvable dihedral group $D_8$ is not fully solvable, since none of the three maximal subgroups is fully invariant. Using the computer algebra system GAP~\cite{GAPnew}, one can show that there are $36$ ultrasolvable groups and $22$ fully solvable groups of order $32$. 

\begin{Pro}
Find a “convenient” characterization of fully solvable groups.
\end{Pro}

\section*{Acknowledgment}
I thank an anonymous referee for pointing out a wrong lemma in a previous version of this paper.


\begin{thebibliography}{1}

\bibitem{BaartmansWoeppel}
A.~H. Baartmans and J. Woeppel, \textit{Groups with a characteristic cyclic
  series}, J. Algebra \textbf{29} (1974), 143--149.

\bibitem{Baer}
R. Baer, \textit{Supersoluble immersion}, Canadian J. Math. \textbf{11} (1959),
  353--369.

\bibitem{CorsiTani}
G. Corsi~Tani, \textit{Su una congettura di J.R. Durbin e M. McDonald}, Atti
  Accad. Naz. Lincei Rend. Cl. Sci. Fis. Mat. Nat. (8) \textbf{69} (1980),
  106--110.

\bibitem{DurbinMcDonald}
J.~R. Durbin and M. McDonald, \textit{Groups with a characteristic cyclic
  series}, J. Algebra \textbf{18} (1971), 453--460.

\bibitem{GAPnew}
The GAP~Group, \textit{GAP -- Groups, Algorithms, and Programming, Version
  4.12.2}; 2022, (\url{http://www.gap-system.org}).

\bibitem{Huppert}
B. Huppert, \textit{Endliche {G}ruppen. {I}}, Grundlehren der Mathematischen
  Wissenschaften, Vol. 134, Springer-Verlag, Berlin, 1967.

\bibitem{Kurzweil}
H. Kurzweil and B. Stellmacher, \textit{The theory of finite groups},
  Universitext, Springer-Verlag, New York, 2004.

\bibitem{Lakatos}
P. Lakatos, \textit{On finite {$p$}-groups with cyclic characteristic series},
  Publ. Math. Debrecen \textbf{74} (2009), 187--193.

\bibitem{Laue}
R. Laue, \textit{Zur {C}harakterisierung der {F}ittinggruppe der
  {A}utomorphismengruppe einer endlichen {G}ruppe}, J. Algebra \textbf{40}
  (1976), 618--626.

\end{thebibliography}
\end{document}